\documentclass[12pt]{amsart}
\usepackage{ amsmath, amsthm, amsfonts, amssymb, color}
 \usepackage{mathrsfs}
\usepackage{amsfonts, amsmath}
 \usepackage{amsmath,amstext,amsthm,amssymb,amsxtra}
 \usepackage{txfonts} 
 \usepackage[colorlinks, citecolor=blue,pagebackref,hypertexnames=false]{hyperref}
 \allowdisplaybreaks
 \usepackage{pgf,tikz}
 \usepackage{pst-plot}

\begin{document}
\baselineskip 16pt

\newcommand\RR{\mathbb{R}}
\def\RN {\mathbb{R}^n}
\newcommand{\norm}[1]{\left\Vert#1\right\Vert}
\newcommand{\abs}[1]{\left\vert#1\right\vert}
\newcommand{\set}[1]{\left\{#1\right\}}
\newcommand{\Real}{\mathbb{R}}
\newcommand{\R}{\mathbb{R}}
\newcommand{\supp}{\operatorname{supp}}
\newcommand{\card}{\operatorname{card}}
\renewcommand{\L}{\mathcal{L}}
\renewcommand{\P}{\mathcal{P}}
\newcommand{\T}{\mathcal{T}}
\newcommand{\A}{\mathbb{A}}
\newcommand{\K}{\mathcal{K}}
\renewcommand{\S}{\mathcal{S}}
\newcommand{\Id}{\operatorname{I}}
\newcommand\wrt{\,{\rm d}}
\newcommand\Ad{\,{\rm Ad}}

\def\SL{\sqrt[m]L}
\newcommand{\mar}[1]{{\marginpar{\sffamily{\scriptsize
        #1}}}}
\newcommand{\li}[1]{{\mar{LY:#1}}}
\newcommand{\el}[1]{{\mar{EM:#1}}}
\newcommand{\as}[1]{{\mar{AS:#1}}}
 \newcommand{\comment}[1]{\vskip.3cm
\fbox{%
\color{red}
\parbox{0.93\linewidth}{\footnotesize #1}}}

\newcommand\CC{\mathbb{C}}
\newcommand\NN{\mathbb{N}}
\newcommand\ZZ{\mathbb{Z}}

\renewcommand\Re{\operatorname{Re}}
\renewcommand\Im{\operatorname{Im}}

\newcommand{\mc}{\mathcal}
\newcommand\D{\mathcal{D}}

\newtheorem{thm}{Theorem}[section]
\newtheorem{prop}[thm]{Proposition}
\newtheorem{cor}[thm]{Corollary}
\newtheorem{lem}[thm]{Lemma}
\newtheorem{lemma}[thm]{Lemma}
\newtheorem{exams}[thm]{Examples}
\theoremstyle{definition}
\newtheorem{defn}[thm]{Definition}
\newtheorem{rem}[thm]{Remark}

\numberwithin{equation}{section}
\newcommand\bchi{{\chi}}

\title[Weak type $(1,1)$   bounds   for   Schr\"odinger  groups ]
{Weak type $(1,1)$    bounds for   Schr\"odinger  groups
 }

\author[P. Chen, X.T. Duong, J. Li, L. Song and L.X. Yan]{Peng Chen, \ Xuan Thinh Duong, \
 Ji Li, \ Liang Song \ and \ Lixin Yan
}
\address{Peng Chen, Department of Mathematics, Sun Yat-sen (Zhongshan)
University, Guangzhou, 510275, P.R. China}
\email{chenpeng3@mail.sysu.edu.cn}
\address{Xuan Thinh Duong, Department of Mathematics, Macquarie University, NSW 2109, Australia}
\email{xuan.duong@mq.edu.au}
\address{Ji Li, Department of Mathematics, Macquarie University, NSW, 2109, Australia}
\email{ji.li@mq.edu.au}
\address{Liang Song, Department of Mathematics, Sun Yat-sen (Zhongshan)
University, Guangzhou, 510275, P.R. China}
\email{songl@mail.sysu.edu.cn}
\address{
Lixin Yan, Department of Mathematics, Sun Yat-sen (Zhongshan) University, Guangzhou, 510275, P.R. China
and Department of Mathematics, Macquarie University, NSW 2109, Australia}
\email{mcsylx@mail.sysu.edu.cn
}

  \date{\today}
 \subjclass[2010]{42B37, 35J10,  47F05}
\keywords{ Weak type $(1,1)$ bounds, Schr\"odinger  group,   Gaussian upper bounds,  space of homogeneous type}

\begin{abstract}
Let $L$ be a non-negative self-adjoint operator acting on $L^2(X)$
where $X$ is a space of homogeneous type with a dimension $n$. Suppose that
the heat kernel of  $L$
satisfies  a  Gaussian upper bound. It is known that
the operator
 $(I+L)^{-s } e^{itL}$ is bounded on $L^p(X)$ for $s>  n|{1/ 2}-{1/p}| $ and $ p\in (1, \infty)$ (see for example, \cite{CCO, H, Sj}).
The  index $s=  n|{1/ 2}-{1/p}|$  was only  obtained recently in \cite{CDLY, CDLY2},
 and this range of $s$ is sharp since it is precisely the range known in the case when $L$
 is  the Laplace operator $\Delta$ on  $X=\mathbb R^n$ (\cite{Mi1}).
 In this paper we
 establish that for $p=1,$ the operator  $(1+L)^{-n/2}e^{itL}$
is of weak type $(1, 1)$, that is,
  there is a constant $C$, independent of $t$ and $f$ so that
\begin{eqnarray*}
  \mu\Big(\Big\{x: \big|(I+L)^{-n/2 }e^{itL} f(x)\big|>\lambda \Big\}\Big)  \leq C\lambda^{-1}(1+|t|)^{n/2} {\|f\|_{L^1(X)} }, \ \ \ t\in{\mathbb R}
\end{eqnarray*}
(for $\lambda > 0$ when $\mu (X) = \infty$ and   $\lambda>\mu(X)^{-1}\|f\|_{L^1(X)}$ when $\mu (X) < \infty$).
Moreover, we also show the index $n/2$ is sharp when $L$ is the Laplacian on ${\mathbb R^n}$ by providing an example.

Our results are applicable to Schr\"odinger  group for large classes of operators
 including  elliptic
 operators on compact manifolds, Schr\"odinger operators with non-negative potentials
 and
 Laplace operators acting on Lie groups of polynomial growth or irregular non-doubling domains
 of Euclidean spaces.

 \end{abstract}

\maketitle


\section{Introduction }
\setcounter{equation}{0}

Let $L$ be a non-negative self-adjoint operator on the Hilbert space $L^2(X),$ where $X$
 is a metric measure space  with a distance $d$ and a measure  $\mu$. Consider
 the Schr\"odinger equation in $X\times {\mathbb R},$
 \begin{eqnarray*}\label{e1.00}
\left\{
\begin{array}{ll}
  i{\partial_t u } + L u=0,\\[4pt]
 u|_{t=0}=f
\end{array}
\right.
\end{eqnarray*}
with initial data $f$. Then
the solution can be formally written as
 \begin{equation}\label{e1.1n}
 u(x,t)=e^{itL}f(x) =   \int_0^{\infty}    e^{it\lambda}dE_L(\lambda) f(x), \ \ \ \ t\in{\mathbb R}
   \end{equation}
 for $f\in L^2(X)$, where   $E_L$ denotes the resolution of the identity associated with  $L.$
   By the spectral theorem (\cite{Mc}),  the operator   $  e^{itL}$  is  continuous on $L^2(X)$,  and forms   the Schr\"odinger group.
   A natural problem is to study the mapping properties
of families of operators derived from the Schr\"odinger group
  on various functional spaces defined on $X$.  This   has attracted a lot of attention in the last   decades,
 and  has been a  very active research topic  in harmonic analysis and partial differential equations-- see for example,
    \cite{A, Bl, Br, BDN, CCO, DN,  EK, H,  H1, JN, JN2, La, Lo, Mi1, O, Sj}.

 In 1960   H\"ormander (\cite{H1}) addressed  this problem  for the
 Laplace operator $\Delta=-\sum_{i=1}^n\partial_{x_i}^2$ on the Euclidean space $\mathbb R^n$,
 and   proved  that   the Schr\"odinger group  $e^{it\Delta}$ is bounded on $L^p({\mathbb R}^n)$ only for
 $p=2$. For $p\not= 2, $
it is well-known (see for example,  \cite{Br, La, Sj}) that for $s > n|{1/ 2}-{1/p}|$, the operator
 $e^{it\Delta}$ maps the Sobolev space $L^p_{2s}(\RN)$ into $L^p(\RN)$. Equivalently, this means that
$(I+\Delta)^{-s } e^{it\Delta}$ is bounded on $L^p(\RN)$, and   this is not the case if  $0<s<  n|{1/ 2}-{1/p}|$.
 The sharp endpoint  $L^p$-Sobolev estimate
 is due to   Miyachi (\cite{Mi1}), which  states that for
every $p\in (1, \infty)$,
\begin{eqnarray}\label{e1.2n}
 \left\|  (1+\Delta)^{-s} e^{it\Delta} f\right\|_{L^p(\mathbb R^n)} \leq C  (1+|t|)^{s}\|f\|_{L^p (\mathbb R^n)},
  \ \ \ t\in{\mathbb R}, \ \ \ s= n\big|{1\over  2}-{1\over  p}\big|
\end{eqnarray}
for some positive constant $C=C(n,p)$ independent of $t$. The estimate \eqref{e1.2n} is sharp both
for the growth in $t$ and for the  derivatives (see \cite[p. 169-170]{Mi1}).


Parallel to the Laplacian on $\mathbb R^n$, the boundedness of the Schr\"odinger group was investigated for different operators
in more general settings.
Depending on the nature of the assumptions regarding the assumption of $e^{-tL}$
and the underlying space $X$, there
 are various nuances of  the mapping properties
of the Schr\"odinger group $e^{-tL}$
on $L^p$ spaces  presently available in the literature.
 For example,
   when $L$ is an elliptic operator of order $m$ with constant coefficients on ${\mathbb R^n}$,
the $L^p$ boundedness of   $(1+ L)^{-s} e^{itL}$ was studied in  \cite{BE, Br, H, Mi1, Sj})  and the references therein;
when $L$ are the sub-Laplacian on Lie groups with polynomial growth and the Laplace-Beltrami operator on
 manifolds with non-negative Ricci curvature, similar results as in \eqref{e1.2n}
for $s> n\left|{1/ 2}-{1/ p}\right|$ and $1<p<\infty$ were first announced by Lohou\'e in \cite{Lo},
then proved by Alexopoulos   in \cite{A}.
   In the abstract setting
of operators on metric measure spaces,   Carron, Coulhon and Ouhabaz \cite{CCO}  showed  that for every $1<p<\infty,$
\begin{eqnarray} \label{e1.555}
 \left\| (I+L)^{-s }e^{itL} f\right\|_{L^p(X)} \leq C (1+|t|)^{s} \|f\|_{L^p(X)}, \ \ \ t\in{\mathbb R}, \ \
  \ s> n\big|{1\over  2}-{1\over  p}\big|,
\end{eqnarray}
  provided
the semigroup $e^{-tL}$, generated by $-L$ on $L^2(X)$ has the kernel  $p_t(x,y)$
which  satisfies
the  Gaussian upper bound \eqref{GE}, i.e.
 $$
 e^{-tL}f(x)=\int_X H_t(x,y)f(y)\, d\mu(y),  \ \  \ {\rm for\ every}\ f\in L^2(X), \ {\rm and\ a.e. \ } x\in X,
  $$
 where $H_t(x,y)$   satisfies
a Gaussian upper bound
\begin{equation*}
 \label{GE}
 \tag{${\rm GE}_m$}
|H_t(x,y)| \leq h_t(x,y)={C\over V(x,t^{1/m})} \exp\left(-c \, {  \left({d(x,y)^{m}\over    t}\right)^{1\over m-1}}\right)
\end{equation*}
for every $t>0, x, y\in X$, where $c, C$ are   positive constants and $m\geq 2.$
Such estimate  \eqref{GE} is typical for elliptic or sub-elliptic differential operators of order $m$
(see for example, \cite{A, BDN, CCO, DN,  D, DM,  DOS, FS,  JN, JN2, O,  Si, Sj, TSC}
and the references therein).

Recently, the problem whether  estimate \eqref{e1.555} holds with  $s =  n|{1/2}-{1/p}|$ was
solved in  \cite[Theorem 1.1]{CDLY}. More specifically, if $L$ satisfies the   Gaussian estimate \eqref{GE}, then
 for every  $p\in (1, \infty)$ there exists a  constant $C=C(n,m, p)>0$ independent of $t$ such that
 \begin{eqnarray} \label{e1.5}
 \left\| (I+L)^{-s }e^{itL} f\right\|_{L^p(X)} \leq C (1+|t|)^{s} \|f\|_{L^p(X)}, \ \ \ t\in{\mathbb R}, \ \
  \ s= n\big|{1\over  2}-{1\over  p}\big|.
\end{eqnarray}
 For $p=1$, it was shown in \cite{CDLY2} that the operator $(I+L)^{-n/2 }e^{itL} $ is bounded from $H^1_L(X)$ into
$L^1(X)$ where $H^1_L(X)$ is a class of Hardy spaces associated to the operator $L$.  Based on this result,
together with interpolation, a different proof of
  the estimate \eqref{e1.5} was obtained.

The aim of this article is to investigate  the Schr\"odinger  groups $e^{itL}$
 on $L^1$ spaces
to  show that under the assumption of  a Gaussian upper bound  \eqref{GE} of $L$,
the operator
$(I+L)^{-n/2 }e^{itL}$ is of weak-type $(1,1)$.  Our result can be stated as  follows.

\begin{thm}\label{th1.1}
Suppose  that $(X, d, \mu)$ is  a  space of homogeneous type  with a dimension $n$.  Suppose that $L$
satisfies the property \eqref{GE}.
Then the operator  $(I+L)^{-n/2}e^{itL}$
  is of weak type $(1, 1)$, that is, there is a constant $C$, independent of $t$ and $f$, so that
  \begin{eqnarray} \label{e1.6}
 \mu\Big(\Big\{x: |(I+L)^{-n/2 }e^{itL} f(x)|>\lambda \Big\}\Big)\leq C\lambda^{-1}(1+|t|)^{n/2} {\|f\|_{L^1(X)} }, \ \ \ t\in{\mathbb R}.
\end{eqnarray}
for $\lambda > 0$ when $\mu (X) = \infty$ and   $\lambda>\mu(X)^{-1}\|f\|_{L^1(X)}$ when $\mu (X) < \infty$.
\end{thm}

\medskip


 We
 would like to
 mention that in \cite[Theorem 1.1]{CKS},
  Chanillo,  Kurtz and  Sampason    proved a weak type $(1, 1)$ result for
 the multiplier operator ${\widehat {Tf}}(\xi)= m(\xi) {\widehat f}(\xi)$ where $m$ is  given by
 $$
m(\xi)=\theta(\xi) |\xi|^{-1} e^{i|\xi|^2} {\hat f} (\xi)
 $$
 on ${\mathbb R}$,
and  $\theta$ is a $C^{\infty}$ function satisfying $\theta(\xi)=1$ for $|\xi|\geq 1$ and  $\theta(\xi)=0$ for $|\xi|\leq 1/2$.
This implies  that
   the operator $(I-{d^2/dx^2})^{-1/2 }e^{-i{d^2/dx^2}} $ is of weak type $(1,1)$, that is,
 for every $\lambda>0,$
 \begin{eqnarray}  \label{e1.8}
  \Big| \Big\{x: \big|(I-{d^2/dx^2})^{-1/2 }e^{-i{d^2/dx^2}} f(x)\big|>\lambda \Big\} \Big| \leq {C\over \lambda} \|f\|_{L^1({\mathbb R})}.
\end{eqnarray}
An examination of their proofs shows    the dimension may be greater than $1$.
As pointed out  in \cite[p. 129]{CKS},
standard arguments which work  for the classical Calder\'on-Zygmund kernels
can not be used  to show the estimate  \eqref{e1.8} since they will fail for large intervals.
Indeed, the kernels $K_{(I-{d^2/dx^2})^{-1/2 }e^{-i{d^2/dx^2}}}(x)$
are not integrable away from the original, and    they
  do not satisfy a regularity condition
 \begin{eqnarray}  \label{e1.9}
 \int_{|x|\geq 2|y|^{1/(1-\theta)}} |K(x-y)-K(x)|dx\leq C<\infty
 \end{eqnarray}
 for any $0<\theta<1$ (\cite{CKS, JS}). To overcome it, they based their proof on the argument used to prove  the Bochner-Riesz multipliers
in the work of Fefferman (\cite[Theorem 3]{F}).

Our proof of    Theorem~\ref{th1.1}
  is different from that of    Chanillo-Kurtz-Sampson \cite{CKS}
where their result  relies   heavily on Fourier analysis. In our setting,   we do not have  Fourier transform  at our disposal.
We also do not assume that the heat kernel $p_t(x,y)$
	  satisfies the standard regularity condition, thus standard techniques of
Calder\'on--Zygmund theory (\cite{St2}) are not applicable.  The
lack  of smoothness of the kernel
will be overcome in the proof   by using the theory of
 singular integrals with rough kernels, which
  lies beyond the scope of the standard Calder\'on-Zygmund
theory (see for example,
    \cite{   DM, DOS, DR} and the references therein).
In the proof of Theorem~\ref{th1.1}, one of the main ingredients is to show that for $k\geq 0$ and $c_1>1,$
\begin{eqnarray}\label{e1.10}
\int_{d(x,y)>  c_1\sqrt{1+|t|}2^k} \big|K_{e^{itL}F_k(L)}(x,y)\big| d\mu(y)\leq C(1+|t|)^{n/2}
\end{eqnarray}
with a constant $C>0$ independent of $k$ and $t,$
where
 $F_k(L) =(1+L)^{-n/2}(1-e^{-2^{mk}L})\varphi_0(2^{-m(k-k_0)/(m-1)}L)$, and
 $\varphi_0$ is a smooth function  with supp $\varphi_0\subset [0,1]$,  $\varphi_0(\lambda)=1$ on $[0,1/2]$,
and $2^{k_0}\sim  \sqrt{1+|t|}$ (see Lemma~\ref{le3.2} below), and
this   is  a crucial estimate in the proof of Theorem~\ref{th1.1}.

The paper is organized as follows.
In Section 2 we provide some preliminary results  on the kernel estimates for the  operators related to
$(I+tL)^{-n/2}$, which play an important  role in the proof of   Theorem~\ref{th1.1}.
The  proof of Theorem~\ref{th1.1}  will be given in Section 3. In Section 4  we discuss some extensions
of Theorem~\ref{th1.1} to measurable
subsets of a space of
homogeneous type and obtain similar results for
operators  on irregular domains with Dirichlet boundary conditions.

\bigskip

\section{ Preliminary results }
\setcounter{equation}{0}

We start by introducing  some notation and assumptions.  Throughout this paper,
unless we mention the contrary, $(X,d,\mu)$ is a metric measure  space where $\mu$
is a Borel measure with respect to the topology defined by the metric $d$.
Next, let
$B(x,r)=\{y\in X,\, {d}(x,y)< r\}$ be  the open ball
with centre $x\in X$ and radius $r>0$. To simplify notation we often just use $B$ instead of $B(x, r)$ and
given $\lambda>0$, we write $\lambda B$ for the $\lambda$-dilated ball
which is the ball with the same centre as $B$ and radius $\lambda r$. Let $B^c$ be the set $X\backslash B$.
We set $V(x,r)=\mu(B(x,r))$ the volume of $B(x,r)$ and we say that $(X, d, \mu)$ satisfies
 the doubling property (see Chapter 3, \cite{CW})
if there  exists a constant $C>0$ such that
\begin{eqnarray}
V(x,2r)\leq C V(x, r)\quad \forall\,r>0,\,x\in X. \label{e1.1}
\end{eqnarray}
If this is the case, there exist  $C, n$ such that for all $\lambda\geq 1$ and $x\in X$
\begin{equation}
V(x, \lambda r)\leq C\lambda^n V(x,r). \label{e1.2}
\end{equation}
In Euclidean space $\RN$ with Lebesgue measure, the parameter $n$ corresponds to
the dimension of the space, but in our more abstract setting, the optimal $n$
 need not even be an integer. There also
exist   $C>0$ and $0\leq D\leq n$ so that
\begin{equation}
V(y,r)\leq C\Big( 1+\frac{d(x,y)}{r}\Big)^D V(x,r)
\label{e1.3}
\end{equation}

\noindent
uniformly for all $x,y\in X$ and $r>0$. Indeed, the property (\ref{e1.2}) with $D=n$ is a direct
consequence of the triangle inequality for the metric
$d$ and the strong homogeneity property (\ref{e1.2}).


  For $1\le p\le+\infty$, we denote the
norm of a function $f\in L^p(X,{\rm d}\mu)$ by $\|f\|_p$, by $\langle \cdot,\cdot \rangle$
the scalar product of $L^2(X, {\rm d}\mu)$, and if $T$ is a bounded linear operator from $
L^p(X, {\rm d}\mu)$ to $L^q(X, {\rm d}\mu)$, $1\le p, \, q\le+\infty$, we write $\|T\|_{p\to q} $ for
the  operator norm of $T$.
Given a  subset $E\subseteq X$, we  denote by  $\chi_E$   the characteristic
function of   $E$.
We denote the dilation of a function $F$ by $\delta_r F(\cdot):=F(r\cdot)$ and
$\widehat{f}\,$ denotes the Fourier  transform, i.e.
of $f$,
$$
\widehat{f}(\xi)={1\over (2\pi)^{n/2}}\int_{\mathbb R^n} f(x)e^{-ix\xi} dx, \ \ \ \ \xi\in \mathbb R^n.
$$
Sometimes we also use  $\widehat{f}$ for ${\mathcal F} f$.

We now   state  the following auxiliary result (see also \cite[Theorem 1]{O}).

\begin{lemma}\label{le2.1}
Assume  that $(X, d, \mu)$ is a space of homogeneous type with a dimension $n$.
 Suppose that $L$
satisfies the Gaussian upper bound \eqref{GE}. Then for $s>n/m$,
 the kernels $K_{(I +tL)^{-s} }(x,y)$ of the operators $ (I +tL)^{-s} $
satisfy
\begin{eqnarray} \label{e2.1}
 \left|K_{(I +tL)^{-s} }(x,y)\right | \leq   a_t(x,y)={ C \over V(x,  t^{1/m})}
  \exp\left(-c \, {  \left({d(x,y)^{m}\over    t}\right)^{1\over 2(m-1)}} \right)
\end{eqnarray}
for all $x,y\in{X}$ and $t>0$.
\end{lemma}

\begin{proof} Note that
\begin{eqnarray*}
  (I+ tL)^{-s}={1\over  \Gamma(s)} \int\limits_{0}^{\infty}e^{- \lambda tL}e^{-\lambda } \lambda ^{s-1} d\lambda.
\end{eqnarray*}
Hence    the kernel $K_{(I +tL)^{-s} }(x,y)$ of $ (I +tL)^{-s} $ is given by
\begin{eqnarray*}
\left|K_{(I +tL)^{-s} }(x,y)\right| &\leq& C \int\limits_{0}^{\infty}
 {e^{-\lambda} \lambda^{s-1}\over V(x,(\lambda t)^{1/m})} \exp\left(-c \, {  \left({d(x,y)^{m}\over    \lambda t}\right)^{1\over m-1}}\right)
 d\lambda .
\end{eqnarray*}
We use  the doubling condition \eqref{e1.2} to obtain that
 $
 V(x, t^{1/m}) \leq C V(x,  (\lambda t)^{1/m}) \big(1+ \lambda^{-1/m}\big)^n
 $ for every $x\in X, \lambda, t>0.$
This, in combination with  the fact  that
$$
\lambda^{1/(m-1)} + \left({d(x,y)^{m}\over    \lambda t}\right)^{1\over m-1}\geq
 2 \left({d(x,y)^{m}\over    t}\right)^{1\over 2(m-1)},
 $$
 gives
 \begin{eqnarray*}
\left|K_{(I +tL)^{-s} }(x,y)\right| &\leq& C  \exp\left(-c_1 \, {  \left({d(x,y)^{m}\over    t}\right)^{1\over 2(m-1)}} \right)
\int\limits_{0}^{\infty}
 {e^{-\lambda + 2^{-1}\lambda^{1/(m-1)}} \lambda^{s-1}\over V(x,(\lambda t)^{1/m})}
 d\lambda\\
 &\leq& { C \over V(x,  t^{1/m})}  \exp\left(-c_1\, {  \left({d(x,y)^{m}\over    t}\right)^{1\over 2(m-1)}} \right)
\int\limits_{0}^{\infty}
  e^{-\lambda + 2^{-1}\lambda^{1/(m-1)} } \lambda^{s-1} \big(1+ \lambda^{-1/m}\big)^{n}
 d\lambda.
\end{eqnarray*}
A simple calculation shows
\begin{eqnarray}\label{e2.2}
\int\limits_{0}^{\infty}
  e^{-\lambda + 2^{-1}\lambda^{1/(m-1)}} \lambda^{s-1} \big(1+ \lambda^{-1/m}\big)^{n}
 d\lambda&\leq& C\int\limits_{0}^{1}
   \lambda^{s-1}  \lambda^{-n/m} d\lambda
  + C\int\limits_{1}^{\infty}
  e^{-\lambda/2} \lambda^{s-1}
 d\lambda,
\end{eqnarray}
and the above  integral \eqref{e2.2} is finite for $m\geq 2$ and $s>n/m$. This shows the desired result \eqref{e2.1}.
\end{proof}

Following \cite[Proposition 2.5]{DR},  we say that a  non-negative function $G_t(x,y)$   satisfies
a Harnack-type inequality  if   for each $\mu>0$ there are $C, \theta\geq 1$ such that
\begin{eqnarray}\label{e2.3}
\sup_{z\in B(y, r)} G_t(x, z)\leq C \inf_{z\in B(y, r)} G_{\theta t}(x, z)
\end{eqnarray}
uniformly for $x, y\in X$ and $r, t>0$ with $r^m\leq \mu t.$
Examples of functions satisfying    \eqref{e2.3} include
  functions $h_t(x,y)$ in \eqref{GE} and   $a_t(x,y)$ defined by  \eqref{e2.1} (see for example,  \cite[Lemma 1]{DM}).

In our Theorem~\ref{th1.1}, we will need some kernel estimates of the operator $(I +tL)^{-s}$ for $s=n/2$.
Since this integral \eqref{e2.2} is infinite for  $m=2$ and $s=n/2$,
 estimate \eqref{e2.3} may or may not hold for  kernel of the operator $(I +tL)^{-n/2}$. Instead,  we have the following
observation,  which will  play a key role in the proof of Theorem~\ref{th1.1} in Section 3.

\begin{lemma}\label{le2.2} Assume  that $(X, d, \mu)$ is a space of homogeneous type with a dimension $n$.
 Suppose that $L$
satisfies the Gaussian upper bound \eqref{GE}. Then
 the kernels $K_{(I +tL)^{-n/2} }(x,y)$ of the operators $ (I +tL)^{-n/2} $
satisfy
\begin{eqnarray*}
 \left|K_{(I +tL)^{-n/2} }(x,y)\right | \leq   P_t(x,y)
\end{eqnarray*}
for all $x,y\in{X}$ and $t>0$ where $ P_t(x,y)$  is a function satisfying
\begin{itemize}
\item[(i)] there exists a constant $C>0$ such that  for $y\in X$ and $t>0$,
$$
 \sup_y\int_X |P_t(x,y)| \, d\mu(x)\leq C, \ \ \ {\rm and}\ \ \ \int_X |P_t(x,y)|^2 \, d\mu(x)\leq {C\over V(y, t^{1/m}) };
 $$

 \item[(ii)] There is a  constant $C>0$ such that
\begin{eqnarray} \label{e2.4}
\sup_{y:\ R\leq d(x,y)\leq 2R} P_t(x,y)\leq C\inf_{y:\ R\leq d(x,y)\leq 2R} P_t(x,y)
\end{eqnarray}
uniformly for $x\in X$ and $R, t>0$.
\end{itemize}
\end{lemma}

\begin{proof} Note that for every $t>0,$
\begin{eqnarray*}
  (I+ tL)^{-n/2}={1\over  \Gamma(n/2)} \int\limits_{0}^{\infty}e^{- \lambda tL}e^{-\lambda } \lambda ^{n/2-1} d\lambda.
\end{eqnarray*}
This, together with the fact \eqref{e1.3} that
\begin{eqnarray*}
 |H_t(x,y)| &\leq& h_t(x,y)={C\over V(x,t^{1/m})} \exp\left(-c \, {  \left({d(x,y)^{m}\over    t}\right)^{1\over m-1}}\right)\\
&\leq& {C\over V(y,t^{1/m})} \left( 1+ {d(x,y)\over t^{1/m} }\right)^n \exp\left(-c \, {  \left({d(x,y)^{m}\over    t}\right)^{1\over m-1}}\right)\\
&\leq& {C\over V(y,t^{1/m})}
 \left( 1 + {d(x,y)^{m}\over    t}\right)^{-(n+1)/m}=:  p_t(x,y),
\end{eqnarray*}
yields that   the kernel $K_{(I +tL)^{-n/2} }(x,y)$ of $ (I +tL)^{-n/2} $ is given by
\begin{eqnarray*}
\left|K_{(I +tL)^{-n/2} }(x,y)\right| &\leq& C \int\limits_{0}^{\infty}p_{\lambda t}(x,y)
e^{-\lambda} \lambda^{n/2-1} d\lambda  =:P_t(x,y)
\end{eqnarray*}
for all $x,y\in{X}$  and $t>0$.
A simple calculation shows that for all $y\in X,$
\begin{eqnarray*}\label{e4.4}
 \int_X |P_t(x,y)| d\mu(x)&\leq & C\int\limits_{0}^{\infty}e^{-\lambda} \lambda^{n/2-1} \left(
 \int_X
 p_{\lambda t}(x,y) d\mu(x) \right)d\lambda
 \leq   C\int\limits_{0}^{\infty}e^{-\lambda} \lambda^{n/2-1}  d\lambda  \leq C.
\end{eqnarray*}
Now we use  the doubling condition \eqref{e1.2} to obtain that
 $
 V(y, t^{1/m}) \leq C V(y,  (\lambda t)^{1/m}) \big(1+ \lambda^{-1/m}\big)^n
 $ for every $y\in X, \lambda, t>0$. By   Minkowski's  inequality,
\begin{eqnarray*}
 \left(\int_X |P_t(x,y)|^2 d\mu(x)\right)^{1/2}
 &\leq& C
\int\limits_{0}^{\infty}e^{-\lambda} \lambda^{n/2-1}\|p_{\lambda t}(\cdot, y)\|_{L^2(X)}  d\lambda
\\
 &\leq&  C\int\limits_{0}^{\infty}e^{-\lambda} \lambda^{n/2-1}    V(y,(\lambda t)^{1/m})^{-1/2}  d\lambda \\
 &\leq&  C  V(y, t^{1/m})^{-1/2}
\int\limits_{0}^{\infty}\big(1+ \lambda^{-1/m}\big)^{n/2} e^{-\lambda} \lambda^{n/2-1}  d\lambda
  \leq    C  V(y, t^{1/m})^{-1/2}.
\end{eqnarray*}

To show  (ii), we observe that  for all $x\in X,$  and  $y_1, y_2\in B(x, 2R)\backslash B(x, R), R>0$, one has
$
d(x, y_1)\leq d(x, y_2)+d(y_1, y_2)\leq d(x, y_2) +4R\leq   5 d(x, y_2).
$
This implies
\begin{eqnarray*}
P_t(x,y_2) &=& \int\limits_{0}^{\infty} p_{\lambda t}(x,y_2)
e^{-\lambda} \lambda^{n/2-1} d\lambda  \\
 &\leq&  C5^{ n-1}  \int\limits_{0}^{\infty} p_{\lambda t}(x,y_1)
e^{-\lambda} \lambda^{n/2-1} d\lambda
 =    5^{n-1} P_t(x,y_1).
\end{eqnarray*}
Hence, the desired estimate is valid with $C=5^{n-1}$ uniformly for $x\in X  $ and $R, t>0$.
 \end{proof}

\begin{lemma}\label{le2.3}
Assume  that $(X, d, \mu)$ is a space of homogeneous type with a dimension $n$.
 Suppose that $L$
satisfies the property \eqref{GE}.
Let $\varphi_1$ be a smooth function with supp $\varphi_1\subset [1/2,\infty]$ and $\varphi_1(\lambda)=1$ on $[1,\infty]$. For every
 $c_0\leq 1$ and    $k\in{\mathbb N}$,
   the kernels $K_{(I +L)^{-n/2}(I-e^{-2^{mk}L}) \varphi_1(c_0 L)}(x,y)$ of the operators $ (I +L)^{-n/2}(I-e^{-2^{mk}L})\varphi_1(c_0 L) $
satisfy
\begin{eqnarray*}
 \left|K_{(I +L)^{-n/2}(I-e^{-2^{mk}L}) \varphi_1(c_0 L)}(x,y)\right | \leq    Q (x,y)
\end{eqnarray*}
for all $x,y\in{X}$ and $t>0$ where $ Q (x,y)$  is a function satisfying

\begin{itemize}
\item[(i)] there exists a constant $C>0$ such that
\begin{eqnarray*}
 \sup_x\int_X  Q (x,y) d\mu(y)\leq C
\end{eqnarray*}
uniformly for $c_0\leq 1$;

 \item[(ii)] There is a  constant $C>0$ such that
\begin{eqnarray*}
\sup_{y:\ R\leq d(x,y)\leq 2R}  Q (x,y)\leq C\inf_{y:\ R\leq d(x,y)\leq 2R}  Q (x,y)
\end{eqnarray*}
uniformly for $x\in X$, $R>0$ and $c_0\leq 1$.
\end{itemize}
\end{lemma}

To prove Lemma~\ref{le2.3}, we recall that  in    \cite[Proposition 4.1]{CCO},
Carron, Coulhon and Ouhabaz used some techniques introduced by Davies (\cite{D2})  to show that the upper
Gaussian estimate \eqref{GE}  on $e^{-tL}, t>0, $ extends to a similar estimate on $e^{-zL}$ where
$z$ belongs to  the whole complex right half-plane  $\mathbb{C}_+=\{z\in{\mathbb C}:  {{\rm Re }z}>0 \}$
\begin{eqnarray*}
|H_z(x,y)|\leq {C(\cos\theta)^{-n}\over \left(V(x, ({|z|\over (\cos \theta)^{m-1} })^{1/m} )\, V(y, ({|z|\over (\cos \theta)^{m-1} })^{1/m} )\right)^{1/2}}
 \exp\left(-c \, {  \left({d(x,y)^{m}\over    |z|}\right)^{1\over m-1}}\cos\theta \right)
 \end{eqnarray*}
 for and all $x, y\in X$, where $\theta={\rm arg} \, z$. The pointwise bounds have the following integrated form:
  for any $s\geq 0 $, $R>0$ and $\tau\in {\mathbb R}$,
the kernels $K_{e^{-(1+i\tau) R^{-m}L}} (x,y) $ of the operators
${e^{-(1+i\tau) R^{-m}L}} $ satisfy
\begin{eqnarray}\label{e2.5}
\int_{X} |K_{e^{-(1+i\tau) R^{-m}L}} (x,y) |^2 d(x,y)^s d\mu(x)\leq V(y, 1/R)^{-1} R^{-s} (1+|\tau|)^s.
 \end{eqnarray}
 For a proof, see \cite[Lemma 4.1]{DOS}.

 \medskip

Next  we apply  the estimate  \eqref{e2.5} to obtain   the following result. For its proof, we refer to \cite[Lemma 4.3]{DOS} or \cite[Theorem 7.15]{O}.

 \begin{lemma}\label{le2.4} Suppose  that $(X, d, \mu)$ is  a  space of homogeneous type  with a dimension $n$.  Suppose that $L$
satisfies the property \eqref{GE}.
For any $R>0, s>0$ and $\epsilon>0$, there exists a constant
$C=C(s, \epsilon)$ such that
\begin{eqnarray*}
\int_X \big|K_{F(\sqrt[m]{L})}(x,y)\big|^2 \big(1+Rd(x,y)\big)^{s} d\mu(x)\leq
{C \over V(y, R^{-1})}
 \|\delta_{R} F\|^2_{C^{{s\over 2} +\epsilon}}
\end{eqnarray*}
for all Borel functions $F$ such that supp $F\subseteq [R/4, R].$
\end{lemma}

 \medskip

\begin{proof}[Proof of Lemma~\ref{le2.3}]
Let
$\phi$ be a non-negative $C_c^{\infty}$ function on $\mathbb R$ such
that ${\rm supp}\  \phi \subseteq ({1/4}, 1)$ and
\begin{eqnarray}\label{e2.6}
\sum_{\ell\in {\Bbb Z}}\phi(2^{-\ell}\lambda)=1, \ \ \ \ \ \forall   \lambda>0,
\end{eqnarray}
and let
$\phi_\ell(\lambda)$ denote the function $\phi(2^{-\ell}\lambda)$.
Since $c_0\leq 1,$ one has
$$
(I +L)^{-n/2}(I-e^{-2^{mk}L})\varphi_1(c_0L) =\sum_{j=-1}^{\infty}(I +L)^{-n/2}(I-e^{-2^{mk}L})\varphi_1(c_0L)\phi(2^{-j}L)=:
\sum_{j=-1}^{\infty} F_j(L).
$$
For every $j\in{\mathbb Z}$, we  write   $G_j (\lambda)=F_j(\lambda)e^{2^{-j}\lambda}$. Then we have
 $F_j(L)=G_j(L)e^{-2^{-j}L} $
 and so
 $$
 K_{F_j(L)}(x,y)
=  \int_{X}
K_{G_j(L) }(x,z) K_{ e^{-2^{-j}L}}(z,y) d\mu(z).
 $$
Note that
 $$
 \big(1+{2^{j/m}d(x,y)} \big) \leq  (1+{2^{j/m}d(x,z)} )(1+{2^{j/m}d(z,y)} )
 $$
and from   \eqref{e1.3},
$$
 V(x, 2^{-j/m}) \leq C V(z,  2^{-j/m}) \left(1+   2^{j/m}d(x, z)  \right)^D.
$$
We apply the heat kernel estimate for $e^{-2^{-j}L}$ and two estimates above to obtain
 that
\begin{eqnarray}\label{e2kernel}
 \left| K_{F_j(L)}(x,y)\right|
&\leq& \int_{X}
\big|K_{G_j(L) }(x,z)\big| \frac{1}{V(z, 2^{-j/m})}
\left(1+   2^{j/m}d(z, y)  \right)^{n+1} d\mu(z)\nonumber\\
&\leq&  C V(x, 2^{-j/m})^{-1} \left(1+{2^{j/m}d(x,y)} \right)^{-(n+1)}  \int_{X}
\big|K_{G_j(L) }(x,z) \big| \left(1+   2^{j/m}d(x, z)  \right)^{ n+D+1 } d\mu(z).
\end{eqnarray}
Then set $\widetilde{G}_j(\lambda)=G_j(\lambda^m)$. By the Cauchy-Schwarz  inequality  and Lemma~\ref{le2.4},
\begin{align*}
&\int_{X}
\big|K_{G_j(L) }(x,z) \big| \left(1+   2^{j/m}d(x, z)  \right)^{ n+D+1 } d\mu(z)\\
&= \int_{X}
\big|K_{\widetilde{G}_j(\sqrt[m]{L}) }(x,z) \big| \left(1+   2^{j/m}d(x, z)  \right)^{ n+D+1 } d\mu(z)\\
&\leq \left(\int_{X}
\big|K_{\widetilde{G}_j(\sqrt[m]{L}) }(x,z) \big|^2 \left(1+   2^{j/m}d(x, z)  \right)^{ 3n+2D+3 } d\mu(z)\right)^{1/2} \left(\int_{X}
 \left(1+   2^{j/m}d(x, z)  \right)^{ -n-1 } d\mu(z)\right)^{1/2}\\
 &\leq C V(x, 2^{-j/m})^{-1/2}
 \|\delta_{2^{j/m}} \widetilde{G}_j\|_{C^{2n+D+2}} V(x,2^{-j/m})^{1/2}\\
 &\leq C
 \|\delta_{2^{j/m}} \widetilde{G}_j\|_{C^{2n+D+2}}.
\end{align*}
On the other hand,
\begin{eqnarray*}
\|\delta_{2^{j/m}} \widetilde{G}_j\|_{C^{2n+D+2}} &=& \|G_j(2^j\lambda^m)\|_{C^{2n+D+2}}
\leq C \|F_j(2^{j}\lambda^m)e^{\lambda^m}\|_{C^{2n+D+2}}
\\
&=&C\|(1 +2^j\lambda^m)^{-n/2}(1-e^{-2^{mk+j}\lambda^m})\varphi_1(2^jc_0\lambda^m)\phi(\lambda^m)\|_{C^{2n+D+2}}\\
&\leq& C\|(1 +2^j\lambda^m)^{-n/2}\phi(\lambda^m)\|_{C^{2n+D+2}}  \|(1-e^{-2^{mk+j}\lambda^m})\phi(\lambda^m)\|_{C^{2n+D+2}}
\|\varphi_1(2^jc_0\lambda^m)\phi(\lambda^m)\|_{C^{2n+D+2}}\\
&\leq& C(1+2^{j})^{-n/2}
\end{eqnarray*}
and so it follows from \eqref{e2kernel} that
\begin{eqnarray*}
|K_{F_j(L)}(x,y)|&\leq& C  (1+2^{j})^{-n/2}{V(x,2^{-j/m})}^{-1}\left(1+2^{j/m}d(x,y)\right)^{-(n+1)}.
\end{eqnarray*}
Therefore   the kernel $K_{(I+L)^{-n/2}(I-e^{-2^{mk}L}) \varphi_1(c_0L)}(x,y)$ of $ (I +L)^{-n/2}(I-e^{-2^{mk}L})\varphi_1(c_0L) $
satisfies
\begin{eqnarray*}
\left|K_{(I +L)^{-n/2}(I-e^{-2^{mk}L})\varphi_1(c_0L) }(x,y)\right |
&\leq& \sum_{j=-1}^{\infty}  |K_{F_j(L)}(x,y)|\\
&\leq& \sum_{j=-1}^{\infty} C(1+2^{j})^{-n/2} \frac{1}{V(x,2^{-j/m})}\left(1+2^{j/m}d(x,y)\right)^{-(n+1)}
 =:  Q (x,y).
\end{eqnarray*}
for all $x,y\in{X}$  and $t>0$.
Obviously,
\begin{eqnarray*}
 \sup_x\int_X | Q (x,y)| d\mu(y)\leq C;
\end{eqnarray*}
 and the similar argument as in (ii) of Lemma~\ref{le2.2} shows
\begin{eqnarray*}
\sup_{y:\ R\leq d(x,y)\leq 2R}  Q (x,y)\leq C\inf_{y:\ R\leq d(x,y)\leq 2R}  Q (x,y)
\end{eqnarray*}
uniformly for $x\in X$ and $R>0$.
\end{proof}

\bigskip

\section{Proof of Theorem~\ref{th1.1}  }
\setcounter{equation}{0}

   Fix a $f\in L^1$ and $\lambda> \mu(X)^{-1}\|f\|_1$,
we apply  the Calder\'on-Zygmund decomposition at height $\lambda$
to $|f|.$ There  exist  constants $C$ and  $K$ so that
\begin{enumerate}
\item[(i)] $
f=g+b=g+\sum_{j}b_j;
$

\vskip 1pt
\item[(ii)]   $\|g\|_1\leq C\|f\|_1,$
 $\|g\|_{\infty}\leq C\lambda$;

\vskip 1pt
\item[(iii)] $b_j$  is supported in  $B_j$ and  $\# \{j: x\in {9\over 8}B_j\} \leq K$ for all $x\in X$;

\vskip 1pt
\item[(iv)]
$\int_{X}\ |b_j|d\mu\leq C\lambda \mu(B_j), $  and
$\sum_{j}\mu(B_j) \leq C{\lambda}^{-1} \|f\|_1.$
\end{enumerate}
It is easy to see that (ii) implies that $\|g\|_2^2\leq C \lambda\|f\|_1.$

 Let $r_{B_j}$ be the radius of $B_j$ and let
  $$
  J_k=\big\{j: \, 2^k \le r_{B_j}<2^{k+1}\big\}, \ \ {\rm for}\ k\in{\mathbb Z}.
  $$  Write
\begin{eqnarray*}
f =g+\sum_{j}b_j&=&
 g+\sum_{k\leq k_0}\sum_{j\in J_k}b_j+\sum_{k>k_0}\sum_{j\in J_k}b_j\\
 &=:& g+h_1+h_2,
\end{eqnarray*}
where $k_0$ is an integer such that $2^{k_0}\leq \sqrt{1+|t|}<2^{k_0+1}$.
Then it  is enough to show that there exists a constant $C>0$ independent
of  $\lambda$ and $t$ such that
\begin{equation}\label{e3.1}
  \mu\Big(\Big\{x:  \big|(I+L)^{-n/2 }e^{itL}g(x)\big|>\lambda\Big\}\Big)
  \leq
C\lambda^{-1}(1+|t|)^{n/2}\|f\|_1
\end{equation}
and  such that for $i=1,2,$
\begin{equation}\label{e3.2}
  \mu\Big(\Big\{x: \big|(I+L)^{-n/2 }e^{itL} h_i (x)\big|>\lambda\Big\}\Big)
  \leq
C\lambda^{-1}(1+|t|)^{n/2}\|f\|_1.
\end{equation}

 \medskip

 Note that  by (ii) $\lambda^{-1}\|g\|_2^2\leq C\|f\|_1.$
Hence by   spectral
 theory,
\begin{eqnarray*}
  \mu\Big(\Big\{x: \big|(I+L)^{-n/2 }e^{itL}g(x)\big|>\lambda \Big\}\Big)&\leq&
\lambda^{-2}\|(I+L)^{-n/2 }e^{itL}g\|_2^2\leq \lambda^{-2}\|g\|_2^2\\
 & \leq
& C\lambda^{-1}\|f\|_1,
\end{eqnarray*}
which proves \eqref{e3.1}.

\medskip

\noindent
{\underline{Proof of  (\ref{e3.2})   for     $i=1$.}}
Since the Schr\"odinger group $e^{-itL}$ is bounded on $L^2(X)$, we have
\begin{eqnarray}\label{e3.3}
  \mu\Big(\Big\{x: \big|(I+L)^{-n/2 }e^{itL}h_1(x)\big|>\lambda\Big\}\Big)&\leq&
\lambda^{-2}\|e^{itL}(I+L)^{-n/2 } h_1\|_2^2\nonumber\\
&\leq&
\lambda^{-2}\|(I+L)^{-n/2 }h_1\|_2^2.
\end{eqnarray}
Recall that $h_1=\sum_{k\leq k_0}\sum_{j\in J_k}b_j$ with $\supp b_j\subset B_j$.
In the following, we set $B^{\ast}_j={9\over 8}B_j$, a ball  with the same center as $B_j$ but expand $9/8$ times. One has
\begin{eqnarray}\label{e3.4}
\lambda^{-2}\Big\|(I+L)^{-n/2 }h_1\Big\|_2^2 &\leq&
\lambda^{-2} \Big\|\sum_{k\leq k_0}\sum_{j\in J_k}  \chi_{B^{\ast}_j}  (I+L)^{-n/2 }b_j \Big\|_2^2 \nonumber\\
&+& \lambda^{-2} \Big\|\sum_{k\leq k_0}\sum_{j\in J_k}  \chi_{(B^{\ast}_j)^c} (I+L)^{-n/2 }b_j\Big\|_2^2
 =  I+II.
\end{eqnarray}

\medskip

Let us estimate    the term $I$.
By Lemma~\ref{le2.2}, the kernels $K_{(I +L)^{-n/2} }(x,y)$ of the operators $ (I +L)^{-n/2} $
satisfy
\begin{eqnarray*}
 \left|K_{(I +L)^{-n/2} }(x,y)\right | \leq   C P_1(x,y)
\end{eqnarray*}
for all $x,y\in{X}$ where $ P_1(x,y)$  is a non-negative function satisfying the properties   in Lemma~\ref{le2.2}.
Since the $B_j^{\ast}$'s have bounded overlaps, we apply Minkowski's inequality and (i)  in Lemma~\ref{le2.2} to obtain
\begin{eqnarray}\label{e3.5}
\left\|\sum_{k\leq k_0}\sum_{j\in J_k}  \chi_{B^{\ast}_j}  (I+L)^{-n/2 }b_j \right\|_2^2 &\leq& C
 \sum_{k\leq k_0}\sum_{j\in J_k}  \int_{B^{\ast}_j} \left| (I+L)^{-n/2 }b_j(x)\right|^2 d\mu(x)\nonumber\\
&\leq& C
 \sum_{k\leq k_0}\sum_{j\in J_k}  \left(\int_X  \|P_1(x,y)\|_{L^2_x}     |b_j(y) | d\mu(y)\right)^2\nonumber\\
&\leq& C
  \sum_{k\leq k_0}\sum_{j\in J_k}  \left(\int_X  V(y,1)^{-1/2}     |b_j(y)| d\mu(y)\right)^2.
\end{eqnarray}
This, in combination with  the doubling condition \eqref{e1.2}  that for every $y\in B_j,$ with $r_{B_j}\leq 2^{k_0}<2\sqrt{1+|t|},$
\begin{eqnarray*}
{1\over V(y,1)}&= &{V(x_{B_j}, 10 r_{B_j})\over V(y, 1)}\cdot {1\over  V(x_{B_j}, 10 r_{B_j})} \leq {V(y, 11 r_{B_j})\over V(y, 1)}\cdot
{1\over \mu(B_j)} \\
&\leq&      {V(y,   r_{B_j})\over V(y, 1)}\cdot
{C\over \mu(B_j)} \leq
{C(1+|t|)^{n/2}\over \mu(B_j)},
\end{eqnarray*}
 yields
\begin{eqnarray*}
I  \leq  C\lambda^{-2}
   \sum_{k\leq k_0}\sum_{j\in J_k} \frac{(1+|t|)^{n/2}}{\mu(B_j)} \left(\int    |b_j(y) | d\mu(y)\right)^2
 &\leq&  C(1+|t|)^{n/2}
\lambda^{-1}  \sum_{k\leq k_0}\sum_{j\in J_k}  \int  |b_j(x)| d\mu(x)\\
&\leq& C(1+|t|)^{n/2}\lambda^{-1}\|f\|_1,
\end{eqnarray*}
 where in the last inequality we used   (iv).

\medskip

Next we estimate  the term $II$.
  Let $x\in (B^{\ast}_j)^c$ and $y\in B_j$,  $ P_1(x,y) $ is equivalent to $ \inf_{y\in {B_j}} P_1(x,y)$
where $P_t(x,y)$ is the function given in Lemma~\ref{le2.2}.
Thus, with $x_{B_j}$ is the center of $B_j$, it follows by Lemma~\ref{le2.2} that
\begin{eqnarray}\label{e3.6}
\sum_{k\leq k_0}\sum_{j\in J_k} \chi_{(B^{\ast}_j)^c}(x) |(1+L)^{-n/2}b_j(x)|
&\leq& C\sum_{k\leq k_0}\sum_{j\in J_k}  \chi_{(B^{\ast}_j)^c}(x)  \int_X P_1(x,y)|b_j(y)|d\mu(y)\nonumber\\
 &\leq& C\sum_{k\leq k_0}\sum_{j\in J_k}   \sup_{y\in B_j} P_1(x,y) \int_X |b_j(y)|d\mu(y)\nonumber\\
 &\leq& C \lambda\sum_{k\leq k_0}\sum_{j\in J_k}   \inf_{y\in B_j}  P_1(x,y)\mu(B_j)\nonumber\\
 &\leq& C\lambda\sum_{k\leq k_0}\sum_{j\in J_k} \int_{B_j}  P_1(x,y) d\mu(y)
 \leq
   C\lambda \int_{X} P_1(x,y) d\mu(y)
 \leq  C\lambda.
\end{eqnarray}
This implies
\begin{eqnarray*}
II&\leq& C
 \lambda^{-2} \int_X \left|\sum_{k\leq k_0}\sum_{j\in J_k}  \chi_{(B^{\ast}_j)^c}(x) (I+L)^{-n/2 }b_j(x)\right|^2 d\mu(x)\nonumber\\
 &\leq& C
 \lambda^{-1} \sum_{k\leq k_0}\sum_{j\in J_k} \int_X \left| (I+L)^{-n/2 }b_j(x)\right| d\mu(x)\nonumber\\
 &\leq& C
 \lambda^{-1} \sum_{k\leq k_0}\sum_{j\in J_k} \int_X  |  b_j (y) | d\mu(y)
  \leq  C
 \lambda^{-1} \|f\|_1.
\end{eqnarray*}
Combining the estimates for $I$ and $II$, we obtain  \eqref{e3.2}    for     $i=1$ with  the term $h_1,$ i.e.,
$$
  \mu\Big(\Big\{x: \Big|(I+L)^{-n/2 }e^{itL}\big(h_1\big)(x)\Big|>\lambda\Big\}\Big)
  \leq
C\lambda^{-1}(1+|t|)^{n/2}\|f\|_1.
$$

\medskip

\medskip

\noindent
{\underline{Proof of  (\ref{e3.2})   for     $i=2$.}}
Set $\Omega_t:=\cup_{j}\sqrt{1+|t|}B_j^\ast$. By (iv),  we have
\begin{eqnarray} \label{mm1}
 \mu\Big(\Big \{x\in \Omega_t: \Big|(I+L)^{-n/2 }e^{itL}h_2 (x)\Big|>\lambda\Big\}\Big)
&\leq& C\mu\Big(\cup_{j}\sqrt{1+|t|}B_j^\ast\Big)\nonumber\\
&\leq& C\sum_j (1+|t|)^{n/2}\mu(B_j^\ast)\nonumber\\
&\leq& C(1+|t|)^{n/2}\sum_j \mu(B_j)
 \leq  C(1+|t|)^{n/2}\lambda^{-1}\|f\|_1.
\end{eqnarray}

 Next we show that
\begin{eqnarray} \label{e3.7}
\mu\Big(\Big\{x\in \Omega_t^c: \Big|(I+L)^{-n/2 }e^{itL} h_2 (x)\Big|>\lambda\Big\}\Big)\leq C(1+|t|)^{n/2}\lambda^{-1}\|f\|_1.
\end{eqnarray}
Since for every $j\in J_k, k>k_0$, the function $b_j$ is supported in $B_j$,  and the radius of the ball  $B_j$
is equivalent to $2^k.$ We then decompose
\begin{eqnarray}\label{eww}
(I+L)^{-n/2 }e^{itL}b_j=(I+L)^{-n/2 }e^{itL}e^{-2^{mk}L}b_j + (I+L)^{-n/2 }e^{itL}(I-e^{-2^{mk}L})b_j.
\end{eqnarray}
To analyse the term  $(I+L)^{-n/2 }e^{itL}e^{-2^{mk}L}b_j,$  we need the following result.

\begin{lemma}\label{le3.1}
With the notation above, there exists a constant $C>0$ such that
$$
\left\|\sum_{k> k_0}\sum_{j\in J_k} e^{-2^{mk}L} b_j\right\|_2\leq C\lambda^{1/2}\|f\|_1^{1/2}.
$$
\end{lemma}

\begin{proof}
The  proof is obtained by using the argument as in \cite[estimate (10), p. 241]{DM}. We give a brief argument
of this proof for completeness and convenience for the reader.

Since the  functions
$
 h_t(x,y) $
   in \eqref{GE} satisfy a Harnack-type inequality  \eqref{e2.3}, one has that for every $j\in J_k,$
\begin{eqnarray*}
\left|e^{-2^{mk}L} b_j(x)\right|&\leq& \int_{X} h_{2^{mk}}(x,y) b_j(y) d\mu(y)\\
&\leq& \|b_j\|_1 \sup_{y\in B_j} h_{2^{mk}}(x,y) \\
&\leq&
 C\lambda \mu(B_j)\inf_{y\in B_j} h_{\theta 2^{mk}}(x,y)
 \leq  C\lambda \int_X   h_{\theta 2^{mk}}(x,y) \chi_{B_j}(y) d\mu(y).
\end{eqnarray*}
Denoting by $M$ the Hardy-Littlewood maximal operator, we then have for every $u\in L^2(X)$
\begin{eqnarray*}
\left|\langle |u|,  e^{-2^{mk}L} b_j\rangle  \right| &\leq&   C\lambda \int_X\int_X |u(x)|   h_{\theta 2^{mk}}(x,y) \chi_{B_j}(y)d\mu(y)d\mu(x)\\
&\leq&  C\lambda \langle M|u|,  \chi_{B_j}\rangle.
\end{eqnarray*}
Since $M$ is bounded on $L^2(X)$, we use the properties (iii) and (iv) to see that
\begin{eqnarray*}
\left\|\sum_{k> k_0}\sum_{j\in J_k} e^{-2^{mk}L} b_j\right\|_2 &\leq&   C\lambda \left\| \sum_{k> k_0}\sum_{j\in J_k}
  \chi_{B_j}\right\|_2\leq C\lambda \left(\sum_j \mu(B_j)\right)^{1/2}\\
  &\leq& C\lambda^{1/2}\|f\|_1^{1/2}.
\end{eqnarray*}
This finishes the proof of Lemma~\ref{le3.1}.
\end{proof}
 \medskip

 \noindent
 {\bf Back to the proof of Theorem~\ref{th1.1}.}\  Now we apply Lemma~\ref{le3.1} to obtain
\begin{eqnarray}\label{ewww}
 \mu\Big(\Big\{x: \Big|\sum_{k>k_0}\sum_{j\in J_k}((I+L)^{-n/2 }e^{itL}e^{-2^{mk}L}b_j)(x)\Big|>\lambda\Big\}\Big)
&\leq& C \lambda^{-2} \Big\|\sum_{k>k_0}\sum_{j\in J_k}(I+L)^{-n/2 }e^{itL}e^{-2^{mk}L}b_j\Big\|_2^2\nonumber\\
&\leq& C  \lambda^{-2} \Big\|\sum_{k>k_0}\sum_{j\in J_k} e^{-2^{mk}L}b_j \Big\|_2^2\nonumber\\
&\leq& C  \lambda^{-1} \|f\|_1.
\end{eqnarray}

Concerning  the term $(I+L)^{-n/2 }e^{itL}(I-e^{-2^{mk}L})b_j$ in \eqref{eww}, we
  let   $\varphi_1$ be a smooth function in Lemma~\ref{le2.3} with supp $\varphi_1\subset [1/2, \infty]$ and $\varphi_1(\lambda)=1$ on $[1, \infty]$, and
let
 $
\varphi_0(\lambda) =1-\varphi_1(\lambda) .
 $
One writes
$$
(I+L)^{-n/2 }e^{itL}(I-e^{-2^{mk}L})b_j=e^{itL}F_k(L)b_j + e^{itL}G_k(L)b_j,
$$
 where
$$
F_k(L)=(I+L)^{-n/2}(I-e^{-2^{mk}L})\varphi_0(2^{-m(k-k_0)/(m-1)}L)
$$
and
$$
G_k(L)=(I+L)^{-n/2}(I-e^{-2^{mk}L})\varphi_1(2^{-m(k-k_0)/(m-1)}L).
$$
Hence
\begin{eqnarray}\label{ewwww}
 \hspace{2cm}  &&\hspace{-3.5cm} \mu\left( \left\{x\in \Omega_t^c: \Big|\sum_{k>0}\sum_{j\in J_k}(I+L)^{-n/2 }
 e^{itL}(I-e^{-2^{mk}L})b_j(x)\Big|>\lambda\right\}\right) \nonumber\\
&\leq& C\mu\left( \left \{x\in \Omega_t^c: \Big|\sum_{k>0}\sum_{j\in J_k}e^{itL}F_k(L)b_j(x)\Big|>\lambda/2 \right\}\right) \nonumber\\
&+&
C\mu\left(\left\{x\in \Omega_t^c: \Big|\sum_{k>0}\sum_{j\in J_k}e^{itL}G_k(L)b_j(x)\Big|>\lambda/2\right\}\right)
 =: III+IV.
\end{eqnarray}

\medskip

 As to be seen later, the term $III$  is the major one.  To do it, we define a Besov type norm of $F$ by
$$
\|F\|_{{{\bf B}^{s}}}:=\int_{-\infty}^{\infty} |\widehat{F}(\tau)|(1+|\tau|)^{s}d\tau,
$$
 where ${\widehat f}$ denotes the Fourier transform of $f$.
Since for every  functions $F$ and $G$,  it can be checked that
 \begin{eqnarray*}
 \|FG\|_{{{\bf B}^{s}}}&=&\int_{-\infty}^{\infty} |\widehat{(FG)}(\tau)|(1+|\tau|)^{s}d\tau\\
 &\leq & \int_{-\infty}^{\infty} \int_{-\infty}^{\infty}\big|  (\widehat{F} (\tau-\eta)   \widehat{G}(\eta) \big|
 |(1+|\tau-\eta|)^{s} |(1+|\eta|)^{s}d\eta  d\tau\\
\end{eqnarray*}
 and so by the Fubini theorem,
$$
\|FG\|_{{{\bf B}^{s}}}\leq \|F\|_{{{\bf B}^{s}}}\|G\|_{{{\bf B}^{s}}}.
$$
We show the following estimate  for the Schr\"odinger group $\{e^{itL}\}_{t\geq 0}$, which plays a key role in the proof
of the term $III$.

\begin{lemma}\label{le3.2}  Suppose  that $(X, d, \mu)$ is  a  space of homogeneous type  with a dimension $n$.  Suppose that $L$
satisfies the property \eqref{GE}.
Let $\varphi_0$ be a smooth function  with supp $\varphi_0\subset [0,1]$ and $\varphi_0(\lambda)=1$ on $[0,1/2]$.
For $k>0$ and $t\in \R$, denote $F_k(L):=(1+L)^{-n/2}(I-e^{-2^{mk}L})\varphi_0(2^{-m(k-k_0)/(m-1)}L)$ where $k_0$ in an integer such that $2^{k_0}\leq \sqrt{1+|t|}<2^{k_0+1}$.
Then for some fixed $c_1>1$,
\begin{eqnarray}\label{e3.8}
\int_{d(x,y)>  c_1\sqrt{1+|t|}2^k} \big|K_{e^{itL}F_k(L)}(x,y)\big| d\mu(y)\leq C(1+|t|)^{n/2}
\end{eqnarray}
with a constant  $C>0$ independent on $k>k_0$ and $t$.
\end{lemma}

\begin{proof}
Let
$\phi$ be an auxiliary  non-negative $C_c^{\infty}$ function on $\mathbb R$ as in \eqref{e2.6}. By the spectral theory,
  we write
\begin{eqnarray}\label{e3.9}
e^{itL}F_k(L) = \sum_{\ell=-\infty}^{m(k-k_0)/(m-1)}e^{itL }(I+L)^{-n/2}(I-e^{-2^{mk}L})\varphi_0(2^{-m(k-k_0)/(m-1)}L)\phi(2^{-\ell}L).
\end{eqnarray}
Set $F_{k,\ell}(L)=e^{itL }(I+L)^{-n/2}(I-e^{-2^{mk}L})\varphi_0(2^{-m(k-k_0)/(m-1)}L)\phi(2^{-\ell}L).$
Then the kernels $K_{F_{k,\ell}(L)}(x,y)$ of the operators $F_{k,\ell}(L)$ satisfies
 $$
 \int_{d(x,y)>c_1\sqrt{1+|t|}2^k} \big|K_{F_{k,\ell}(L)}(x,y)\big| d\mu(y)\leq
  I_{\ell}(x,y),
  $$
where
$$
I_{\ell}(x,y)=\left(\int_{X} \big|K_{F_{k,\ell}(L)}(x,y)\big|^2 (1+2^{\ell/m}d(x,y))^sd\mu(y)\right)^{1/2}
  \left(\int_{d(x,y)> c_1\sqrt{1+|t|}2^k}   (1+2^{\ell/m}d(x,y))^{-s}d\mu(y)\right)^{1/2}
  $$
  for  some  $n<s<n+1/2$.
Next we set  $G(\lambda)=\big(\delta_{2^{\ell}} F_{k,\ell}(\lambda)\big)e^{\lambda}.$
Thus in virtue of the Fourier inversion formula
$$
F_{k,\ell}(L)=G(2^{-\ell}L)e^{-2^{-\ell}L}=\frac{1}{2\pi}\int_{\R} e^{(i\tau-1)2^{-\ell}L}\widehat{G}(\tau) d\tau.
$$
By \eqref{e2.5}, we get
\begin{eqnarray*}
 &&\left(\int_{X} \big|K_{F_{k,\ell}(L)}(x,y)\big|^2 (1+2^{\ell/m}d(x,y))^sd\mu(y)\right)^{1/2}\\
 &\leq&  \frac{1}{2\pi}\int_{\R} \left(\int_{X} \big|K_{e^{(i\tau-1)2^{-\ell}L}}(x,y)\big|^2 (1+2^{\ell/m}d(x,y))^sd\mu(y)\right)^{1/2} |\widehat{G}(\tau)|d\tau\\
 &\leq& CV(x, 2^{-\ell/m})^{-1/2}\int_{\mathbb R} | {\widehat G}(\tau)| (1+|\tau|)^{s/2} d\tau\\
&\leq& CV(x, 2^{-\ell/m})^{-1/2}\|G\|_{{\bf B}^{s/2 } }.
\end{eqnarray*}
To go on, we claim  that
\begin{eqnarray}\label{e3.10}
\|G\|_{{\bf B}^{s/2 } }
&\leq&
C \min\{1,2^{\ell+mk}\}\min\{1,2^{-\ell n/2}\}\max\{1, (2^\ell(1+|t|))^{s/2}\}.
\end{eqnarray}
Let us show the claim \eqref{e3.10}.
Now we let
 $\eta\in C_c^\infty({\mathbb R})$  with  $\supp\eta\subset [1/8, 2]$ and $\eta(\lambda)=1$ for $\lambda\in [1/4,1]$.
One has
\begin{eqnarray*}
&&\|G\|_{{\bf B}^{s/2 }}=\|[\delta_{2^{\ell}} F_{k,\ell}(\lambda)]e^{\lambda}\|_{{\bf B}^{s/2 } }\\
&=& \| \phi(\lambda)e^{it2^{\ell}\lambda}(1+2^\ell \lambda)^{-n/2}(1-e^{-2^{mk+\ell}\lambda})\varphi_0(2^{-m(k-k_0)/(m-1)+\ell}\lambda)e^{\lambda}\|_{{\bf B}^{s/2} }\\
&\leq& \| \eta(\lambda)\varphi_0(2^{-m(k-k_0)/(m-1)+\ell}\lambda)\|_{{\bf B}^{s/2 } }
\| \eta(\lambda)e^{\lambda}(1-e^{-2^{mk+\ell}\lambda})\|_{{\bf B}^{s/2 } }
\| \phi(\lambda)e^{it2^{\ell}\lambda}(1+2^\ell \lambda)^{-n/2}\|_{{\bf B}^{s/2 } }\\
&\leq& C\| \eta(\lambda)\varphi_0(2^{-m(k-k_0)/(m-1)+\ell}\lambda)\|_{{C}^{s/2 +2} }
\| \eta(\lambda)e^{\lambda}(1-e^{-2^{mk+\ell}\lambda})\|_{{C}^{s/2+2 } }
\| \phi(\lambda)e^{it2^{\ell}\lambda}(1+2^\ell \lambda)^{-n/2}\|_{{\bf B}^{s/2 } }.
\end{eqnarray*}
Note that $\ell\leq m(k-k_0)/(m-1)$ and $\supp \eta\subset[1/8,2]$,
$$
\| \eta(\lambda)\varphi_0(2^{-m(k-k_0)/(m-1)+\ell}\lambda)\|_{{C}^{s/2 +2} }\leq C_s
$$
and
$$
\| \eta(\lambda)e^{\lambda}(1-e^{-2^{mk+\ell}\lambda})\|_{{C}^{s/2+2 } }\leq C_s \min\{1,2^{\ell+mk}\}
$$
with $C_s$ independent of $k$ and $\ell$. Let us  estimate $\| \phi(\lambda)e^{it2^{\ell}\lambda}(1+2^\ell \lambda)^{-n/2}\|_{{\bf B}^{s/2 } }$.
The Fourier transform ${\mathcal F} \big(\phi e^{it2^{\ell}\cdot}(1+2^\ell \cdot)^{-n/2}\big)(\tau)$
of $\phi(\lambda)e^{it2^{\ell}\lambda}(1+2^\ell \lambda)^{-n/2}$ is given by
\begin{eqnarray*}
{\mathcal F} \big(\phi e^{it2^{\ell}\cdot}(1+2^\ell \cdot)^{-n/2}\big)(\tau)=\int_{\mathbb{R}} \phi(\lambda)\frac{e^{i( 2^\ell t-\tau)\lambda}}{(1+2^\ell\lambda)^{n/2}}d\lambda.
\end{eqnarray*}
Integration by parts gives for every $N\in \mathbb{N}$,
\begin{eqnarray*}
\left|{\mathcal F} \big(\phi e^{it2^{\ell}\cdot}(1+2^\ell \cdot)^{-n/2}\big)(\tau)\right|\leq
C_N(1+2^{\ell})^{-n/2}(1+|2^\ell t-\tau|)^{-N},
\end{eqnarray*}
which yields for $s>n$,
\begin{eqnarray*}
\|  \phi(\lambda)e^{it2^{\ell}\lambda}(1+2^\ell \lambda)^{-n/2}\|_{{\bf B}^{s/2 } }&\leq& C\min\{1,2^{-\ell n/2}\}\int_{\mathbb{R}} (1+|2^\ell t-\tau|)^{-N}
(1+|\tau|)^{s/2} d\tau\\
&\leq& C\min\{1,2^{-\ell n/2}\}(1+2^\ell t)^{s/2}\\
&\leq& C\min\{1,2^{-\ell n/2}\}\max\{1, (2^\ell(1+|t|))^{s/2}\}.
\end{eqnarray*}
Hence,
\begin{eqnarray*} 
\|G\|_{{\bf B}^{s/2 } }
&\leq&
C \min\{1,2^{\ell+mk}\}\min\{1,2^{-\ell n/2}\}\max\{1, (2^\ell(1+|t|))^{s/2}\}.
\end{eqnarray*}
This proves our claim  \eqref{e3.10}. From this, we see that
\begin{eqnarray*}
I_{\ell}(x,y)\leq  C  \min\{1,2^{\ell+mk}\}\min\{1,2^{-\ell n/2}\}\max\{1, (2^\ell(1+|t|))^{s/2}\}
  (1+ c_1 \sqrt{1+|t|}2^{k+{\ell\over m}})^{(n-s)/2},
 \end{eqnarray*}
where we have used the fact that
\begin{eqnarray*}
 \int_{d(x,y)>c_1\sqrt{1+|t|}2^k}   (1+2^{\ell/m}d(x,y))^{-s}d\mu(y)
 \leq CV(x, 2^{-\ell/m}) (1+ c_1 \sqrt{1+|t|}2^{k+{\ell\over m}})^{n-s}.
\end{eqnarray*}
 See
\cite[Lemma 4.4]{DOS}. Therefore,
\begin{eqnarray*}
 \int_{d(x,y)>c_1\sqrt{1+|t|}2^k} \big|K_{e^{itL}F_k(L)}(x,y)\big| d\mu(y)
 &\leq&
\sum_{\ell=-\infty}^{m(k-k_0)/(m-1)}I_{\ell}(x,y)\\
&\leq& C\sum_{\ell=1}^{m(k-k_0)/(m-1)}(1+|t|)^{(n+s)/4}2^{k(n-s)/2}2^{(1-1/m)(s-n)\ell/2}\\
&& +C\sum_{\ell=-2k_0}^{0} (1+|t|)^{(n+s)/4}2^{k(n-s)/2} 2^{(s+(n-s)/m)\ell/2}\\
&& +C\sum_{\ell=-mk}^{-2k_0} (1+|t|)^{(n-s)/4} 2^{k(n-s)/2} 2^{(n-s)\ell/(2m)} \\
&&+ C\sum_{\ell=-\infty}^{-mk}(1+|t|)^{(n-s)/4} 2^{(mk+\ell)(1+(n-s)/(2m))}\\
&\leq& C(1+|t|)^{n/2}.
\end{eqnarray*}
The proof of Lemma~\ref{le3.2} is complete.
\end{proof}

\medskip

We now apply  Lemma~\ref{le3.2} to estimate the term $III$  to obtain
\begin{eqnarray}\label{mm2}
III&=&C\mu\Big(\{x\in \Omega_t^c: \Big|\sum_{k>k_0}\sum_{j\in J_k}e^{itL}F_k(L)b_j (x)\Big|>\lambda/2\}\Big)\nonumber\\
&\leq&
C\lambda^{-1}\sum_{k>k_0}\sum_{j\in J_k} \int_{\Omega_t^c} |e^{itL}F_k(L)b_j(x)| d\mu(x)\nonumber\\
&\leq& C\lambda^{-1}\sum_{k>k_0}\sum_{j\in J_k} \int_{(\sqrt{1+|t|}B_j^\ast)^c} \int_{X}
|K_{e^{itL}F_k(L)}(x,y)||b_j(y)|d\mu(y) d\mu(x)\nonumber\\
&\leq& C\lambda^{-1}\sum_{k>k_0}\sum_{j\in J_k} \int_{X} \int_{(\sqrt{1+|t|}B_j^\ast)^c}
|K_{e^{itL}F_k(L)}(x,y)|d\mu(x)|b_j(y)| d\mu(y)\nonumber\\
&\leq& C\lambda^{-1}(1+|t|)^{n/2}\sum_{k>k_0}\sum_{j\in J_k} \int_{X} |b_j(y)| d\mu(y)\nonumber\\
&\leq& C\lambda^{-1}(1+|t|)^{n/2}\|f\|_1.
\end{eqnarray}

\medskip

Concerning the term $IV$,
since the Schr\"odinger group $e^{-itL}$ is bounded on $L^2(X)$, we have
\begin{eqnarray}\label{e3.11}
IV&\leq& C
\lambda^{-2}\Big\| \sum_{k>k_0}\sum_{j\in J_k} e^{itL} G_k(L)b_j  \Big\|_2^2
\nonumber\\
&\leq&
C\lambda^{-2}\Big\| \sum_{k>k_0}\sum_{j\in J_k}G_k(L)  b_j\Big\|_2^2 \nonumber\\
&\leq& C\lambda^{-2} \left\|\sum_{k>k_0}\sum_{j\in J_k}  \chi_{B^{\ast}_j}  G_k(L)  b_j \right\|_2^2
 +  C\lambda^{-2} \left\|\sum_{k>k_0}\sum_{j\in J_k}  \chi_{(B^{\ast}_j)^c} G_k(L)  b_j\right\|_2^2
 =  IV_1+IV_2.
\end{eqnarray}
We point out that the terms $IV_1$ and $IV_2$ can be handled similarly to the terms $I$ and $II$ with some modifications, respectively.

To estimate  the  term $IV_1$, we note that
\begin{eqnarray*}
&&\|G_k(L)b_j\|_2\\
&\leq& \|G_k(L)(I+2^{-m(k-k_0)/(m-1)}L)^{n/2}\|_{2\to 2} \|(I+2^{-m(k-k_0)/(m-1)}L)^{-n/2}b_j\|_2\\
&\leq&     \|(1+\lambda)^{-n/2} (1-e^{-2^{mk}\lambda}) \varphi_1(2^{-m(k-k_0)/(m-1)}\lambda)(1+2^{-m(k-k_0)/(m-1)}\lambda)^{n/2}\|_{\infty}
\|(I+2^{-m(k-k_0)/(m-1)}L)^{-n/2}b_j\|_2\\
&\leq&  C2^{-\frac{m(k-k_0)n}{2(m-1)}}\|(I+2^{-m(k-k_0)/(m-1)}L)^{-n/2}b_j\|_2.
\end{eqnarray*}
Since the $B_j^{\ast}$'s have bounded overlaps, we apply Minkowski's inequality and the property (i) in Lemma~\ref{le2.2} to obtain
\begin{eqnarray}\label{e3.12}
\left\|\sum_{k>k_0}\sum_{j\in J_k}  \chi_{B^{\ast}_j}  G_k(L)  b_j \right\|_2^2 &\leq& C
 \sum_{k>k_0}\sum_{j\in J_k}
 \int_{X} \left| G_k(L) b_j(x)\right|^2 d\mu(x)\nonumber\\
 &\leq& C
 \sum_{k>k_0}\sum_{j\in J_k}  2^{-\frac{m(k-k_0)n}{ m-1 }}
 \int_{X} \left| (I+2^{-m(k-k_0)/(m-1)}L)^{-n/2}  b_j(x)\right|^2 d\mu(x)\nonumber\\
&\leq& C
 \sum_{k>k_0}\sum_{j\in J_k} 2^{-\frac{m(k-k_0)n}{ m-1 }} \left(\int_X  \|P_{2^{-m(k-k_0)/(m-1)}}(x,y)\|_{L^2_x}     |b_j(y) | d\mu(y)\right)^2\nonumber\\
&\leq& C
  \sum_{k>k_0}\sum_{j\in J_k} 2^{-\frac{m(k-k_0)n}{ m-1 }} \left(\int_X  V(y,2^{- (k-k_0)/(m-1)})^{-1/2}     |b_j(y)| d\mu(y)\right)^2.
\end{eqnarray}
This, in combination with  the doubling condition \eqref{e1.2}  that for every $y\in B_j,$
\begin{eqnarray*}
{1\over V(y,2^{- (k-k_0)/(m-1)})} =  {V(x_{B_j}, 2^k)\over V(y, 2^{- (k-k_0)/(m-1)})}\cdot {C\over  V(x_{B_j}, 2^k)}
&\leq& {V(y, 2^{k+1})\over V(y, 2^{- (k-k_0)/(m-1)})}\cdot
{C\over \mu(B_j)} \\
&\leq&   C 2^n  (1+|t|)^{n/2} 2^{\frac{m(k-k_0)n}{ m-1}}
{1\over \mu(B_j)},
\end{eqnarray*}
 yields
\begin{eqnarray*}
IV_1  \leq  C\lambda^{-2}
   \sum_{k>k_0}\sum_{j\in J_k} \frac{(1+|t|)^{n/2}}{\mu(B_j)} \left(\int    \left|b_j(y)\right| d\mu(y)\right)^2
 &\leq&  C (1+|t|)^{n/2}
\lambda^{-1}  \sum_{k>k_0}\sum_{j\in J_k}  \int \left|b_j(y)\right| d\mu(y)\\
&\leq& C(1+|t|)^{n/2}\lambda^{-1}\|f\|_1.
\end{eqnarray*}

Now we prove the  term $IV_2$. Let $x\in (B^{\ast}_j)^c$ and $y\in B_j$,  $ Q(x,y) $ is equivalent to
$ \inf_{y\in {B_j}} Q(x,y),$
where $Q (x,y)$ is the function given in Lemma~\ref{le2.3}.
Thus, with $x_{B_j}$ is the center of $B_j$, we have
that for a fixed $x\in (B^{\ast}_j)^c$, it follows from (ii) in Lemma~\ref{le2.3}  that
\begin{eqnarray}\label{e3.13}
\sum_{k>k_0}\sum_{j\in J_k} \chi_{(B^{\ast}_j)^c}(x) |G_k(L) b_j(x)|
&\leq& C\sum_{k>k_0}\sum_{j\in J_k}  \chi_{(B^{\ast}_j)^c}(x)  \int_X Q(x,y)|b_j(y)|d\mu(y)\nonumber\\
 &\leq& C\sum_{k>k_0}\sum_{j\in J_k}   \sup_{y\in B_j} Q(x,y) \int_X |b_j(y)|d\mu(y)\nonumber\\
 &\leq& C \lambda\sum_{k>k_0}\sum_{j\in J_k}   \inf_{y\in B_j}  Q(x,y)\mu(B_j)\nonumber\\
 &\leq& C\lambda\sum_{k>k_0}\sum_{j\in J_k} \int_{B_j}  Q(x,y) d\mu(y)
 \leq
   C\lambda \int_{X} Q (x,y) d\mu(y)
 \leq  C\lambda.
\end{eqnarray}
This implies
\begin{eqnarray*}
IV_2&\leq& C
 \lambda^{-2} \int_X \left|\sum_{k>k_0}\sum_{j\in J_k}  \chi_{(B^{\ast}_j)^c}(x) G_k(L)  b_j (x)\right|^2 d\mu(x)\nonumber\\
 &\leq& C
 \lambda^{-1} \sum_{k>k_0}\sum_{j\in J_k} \int_X \left| G_k(L)  b_j (x)\right| d\mu(x)\nonumber\\
 &\leq& C
 \lambda^{-1} \sum_{k>k_0}\sum_{j\in J_k} \int_X  |  b_j (y) | d\mu(y)
  \leq  C
 \lambda^{-1} \|f\|_1.
\end{eqnarray*}
From  the estimates for $IV_1$ and $IV_2$,
$IV \leq
C\lambda^{-1}(1+|t|)^{n/2}\|f\|_1.
$
This combined with \eqref{mm1}, \eqref{e3.7}, \eqref{eww}, \eqref{ewww}, \eqref{ewwww} and \eqref{mm2} yields
   \eqref{e3.2}    for     $i=2$.
The proof of the theorem  now follows from \eqref{e3.1}  and  \eqref{e3.2} .
\hfill$\square$

\bigskip

Recall that for
every $p\in (1, \infty)$,
    Miyachi (\cite{Mi1}) showed that   there exists a positive constant $C=C(n,p)$ independent of $t$ and $f$ so that
\begin{eqnarray}\label{e4.2n}
 \left\|  (1+\Delta)^{-s} e^{it\Delta} f\right\|_{L^p(\mathbb R^n)} \leq C  (1+|t|)^{s}\|f\|_{L^p (\mathbb R^n)},
  \ \ \ t\in{\mathbb R}, \ \ \ s= n\big|{1\over  2}-{1\over  p}\big|.
\end{eqnarray}
The estimate \eqref{e4.2n} is sharp
for the growth in $t,$ that is, the factor $(1+|t|)^{s}$ can not be improved.
  Following
\cite{Mi1}, we show that  for $p=1,$ the factor $(1+|t|)^{n/2}$ in our Theorem~\ref{th1.1}
is the best possible  in the case that when $L$ is the classical Laplacian on $\R^n$. This can be seen by the following result.

\begin{prop}\label{prop3.1}
 For every $t\geq 1$, there exists  a constant $C=C(n)$ independent of $t$ such that
$$
\|(1+\Delta)^{-n/2}e^{it\Delta}\|_{L^1({\mathbb R^n})\to L^{1,\infty}({\mathbb R^n})}\geq C t^{n/2}.
$$
\end{prop}

\begin{proof}
Let $\phi\in C^\infty(1/2,\infty)$ with $\supp \phi \subset(1/2,\infty)$ and $\phi=1$ on $[1,\infty)$. We define
$$
f=c\mathcal{F}^{-1}(\phi(|\xi|)(1+|\xi|^2)^{-1})
$$
with  some constant $c>0$ so that     $\|f\|_1=1$. Then
$$
(1+\Delta)^{-n/2}e^{it\Delta}f(x)=\mathcal{F}^{-1}(\phi(|\xi|)e^{it|\xi|^2}(1+|\xi|^2)^{-n/2-1})(x).
$$
Following \cite[Lemma 1]{Mi1}, there exist constants $C_1$ and $C_2$ such that for  $|x|\geq C_2t,$
$$
|(1+\Delta)^{-n/2}e^{it\Delta}f(x)|\geq C_1 t|x|^{-n/2-1}, \ \ \  \ t>1.
$$
Then we have
\begin{eqnarray*}
\Big\|(1+\Delta)^{-n/2}e^{it\Delta}f\Big\|_{L^{1,\infty}}&=&\sup_{\lambda>0}\lambda
 \Big|\Big\{x\in \R^n: \big|(1+\Delta)^{-n/2}e^{it\Delta}f(x)\big|>\lambda\Big\}\Big|\\
&\geq& \sup_{\lambda>0}\lambda \Big| \Big\{|x|\geq C_2 t: \big|C_1 t|x|^{-n/2-1}\big|>\lambda\Big\}\Big|.
\end{eqnarray*}
Letting  $\lambda=C_1t^{-n/2}/(1+C_2)^{n/2+1}$, we have
\begin{eqnarray*}
\|(1+\Delta)^{-n/2}e^{it\Delta}f\|_{L^{1,\infty}}
&\geq& \frac{C_1t^{-n/2}}{(1+C_2)^{n/2+1}}\Big|\Big\{|x|\geq C_2 t: |C_1 t|x|^{-n/2-1}|>\frac{C_1t^{-n/2}}{(1+C_2)^{n/2+1}}\Big\}\Big|\\
&\geq& C t^{-n/2}\Big|\Big\{x: C_2 t \leq |x|<(1+C_2)t\Big\}\Big|\\
&\geq& C(1+|t|)^{n/2}.
\end{eqnarray*}
This finishes the proof of Proposition~\ref{prop3.1}.
\end{proof}

Our results are applicable to Schr\"odinger  group for large classes of operators
 including  elliptic
 operators on compact manifolds, Schr\"odinger operators with non-negative potentials
 and
 Laplace operators acting on Lie groups of polynomial growth (see for example, \cite{CDLY, DN, DOS, JN} and the references therein).
For example,
  we consider   the Schr\"odinger operator
 $$
 L=-\Delta+V\ \ \ \ \ \ \  {\rm on\ }  {\mathbb R^n, \ \ \ \ n\geq 1,}
 $$
 where $V: {\Bbb R}^n\rightarrow {\Bbb R}, V\in L^1_{\rm loc}({\Bbb R}^n)$ and
   $V\geq 0.$ The operator $L$ is defined by the quadratic form.
 The
Feynman-Kac formula implies that the semigroup kernels
$p_t(x,y)$, associated to $e^{-tL}$, satisfy the estimates

\begin{equation}\label{e8.4}
0\leq  p_t(x,y)\leq (4\pi t)^{-{n\over 2}}\exp\left(-\frac{|x-y|^2}{4t}\right)\ \  \ {\rm for\
all\ }\  t>0, \ \ x,y\in{\Bbb R}^n.
\end{equation}
See page 195 of \cite{O}.

  \begin{thm}\label{prop3.2}  Assume that
   $L=-\Delta+V$ where $\Delta$ is the standard Laplace operator acting on ${\Bbb R}^n $  and
   $  V\in L^1_{\rm loc}({\Bbb R}^n )$ is a non-negative function.
Then we have

 \begin{itemize}
\item[(i)] The operator $(I+L)^{-n/2 }e^{itL}$ is of weak type $(1,1)$, that is,
there is a constant $C$,  independent of $t$ and $f$ so that  for $\lambda > 0$,
\begin{eqnarray*}
  \left| \{x\in \RN: |(I+L)^{-n/2 }e^{itL} f(x)|>\lambda \}\right|
  \leq C\lambda^{-1}(1+|t|)^{n/2} {\|f\|_{L^1(\RN)} }, \ \ \ t\in{\mathbb R}.
\end{eqnarray*}

\item[(ii)] For every $1<p<\infty$ and $s\geq n|{1/p-1/2}|$, the operator $(I+L)^{s }e^{itL}$ is   bounded on $L^p(\RN)$,
  and there exists a    constant $C$, independent of $t$ and $f$ so that
 \begin{eqnarray*}
 \left\| (I+L)^{-s }e^{itL} f\right\|_{L^p(\RN)} \leq C (1+|t|)^{s} \|f\|_{L^p(\RN)}, \ \ \ t\in{\mathbb R}.
\end{eqnarray*}
\end{itemize}
 \end{thm}

\begin{proof}
This result is a consequence of
 Theorem~\ref{th1.1} and \cite[Theorem 1]{CDLY}.
\end{proof}

 We note that    under suitable  additional assumptions this result
  can be extended by a similar proof to situation of magnetic Schr\"odinger operators
  acting on a complete Riemannian manifold with non-negative potentials.

\bigskip

 \section{Schr\"odinger groups  on measurable
subsets of a space of
homogeneous type}
 \setcounter{equation}{0}

 We assume in this section that $\Omega$ is a measurable
subset of a space of
homogeneous type $(X, d, \mu)$. An example of $\Omega$ is
an open domain
of the Euclidean space ${\mathbb R}^n$.  If $\Omega$ possesses
certain smoothness
on its boundary, for example Lipschitz boundary, then it is a space
of homogeneous type and the results of Sections 2  and 3 are
applicable. However,  a general
measurable set $\Omega$ needs not satisfy the doubling property,
hence it is not a space of homogeneous type. Such a measurable set
$\Omega$ appears
naturally in boundary value problems, for example   partial differential
equations with Dirichlet boundary conditions.

 In this section we are interested in dealing with  Schr\"odinger groups  in those contexts.
 As it is pointed out in \cite{DM}, one can extend the singular operators defined in $\Omega$
 to the space $X.$ Since there is no assumption on the regularity of the kernels in
 space variables, the extension of the kernel still satisfies similar conditions. Given $T$, a bounded linear operator
 on $L^p(\Omega), 1<p<\infty$, the extension of $T$ to $X$ is defined as
\begin{eqnarray}\label{TT}
{\widetilde T}f(x)=
\left\{
\begin{array}{ll} T\big(f\chi_{\Omega}\big)(x),  \ &x\in \Omega\\[6pt]
0, \ &x\not\in \Omega,
\end{array}
\right.
\end{eqnarray}
where $\chi_{\Omega}$ is the characterization function on $\Omega$. Then  $T$ is bounded
 on $L^p(\Omega)$ if and only if ${\widetilde T}$ is bounded  on $L^p(X).$
 Also $T$ is of weak type $(1,1)$ on $\Omega$ if and only if ${\widetilde T}$ is of weak type $(1,1)$ on $X.$
  If $K$ is the kernel of $T$, then the associated kernel of ${\widetilde T}$  is given by
 \begin{eqnarray*}
{\widetilde K}(x,y))=
\left\{
\begin{array}{ll}
K(x,y),  \ &{\rm when}\ x\in\Omega \ {\rm and}\ y\in \Omega\\[6pt]
0, \ & {\rm otherwise}.
\end{array}
\right.
\end{eqnarray*}
 As it is observed in \cite{DM},
  the assumptions on the kernels do not involve their regularity so they imply similar
  properties on the kernels of the extended operators.

 The following result gives an example of Schr\"ordiner groups on spaces without the doubling condition.

  \begin{thm}\label{prop4.1} Suppose that  $\Delta_\Omega$ is the Laplace operator with Dirichlet
  boundary condition $\Omega\subseteq {\Bbb R}^n$.
Then we have

 \begin{itemize}
\item[(i)] The operator $(I+\Delta_\Omega)^{-n/2 }e^{it\Delta_\Omega}$ is of weak type $(1,1)$, that is,
there is a constant $C$,  independent of $t$ and $f$ so that
\begin{eqnarray*}
  \left| \{x\in \Omega: |(I+\Delta_\Omega)^{-n/2 }e^{it\Delta_\Omega} f(x)|>\lambda \}\right|
  \leq C\lambda^{-1}(1+|t|)^{n/2} {\|f\|_{L^1(\Omega)} }, \ \ \ t\in{\mathbb R}
\end{eqnarray*}
 for $\lambda > 0$ when $|\Omega| = \infty$ and   $\lambda> \|f\|_1/|\Omega|$ when $|\Omega|< \infty$.

\item[(ii)] For every $1<p<\infty$ and $s\geq n|{1/p-1/2}|$, the operator $(I+\Delta_\Omega)^{s }e^{it\Delta_\Omega}$ is   bounded on $L^p(\Omega)$,
  and there exists a    constant $C$, independent of $t$ and $f$ so that

 \begin{eqnarray*}
 \left\| (I+\Delta_\Omega)^{-s }e^{it\Delta_\Omega} f\right\|_{L^p(\Omega)}  \leq C (1+|t|)^{s} \|f\|_{L^p(\Omega)} , \ \ \ t\in{\mathbb R}.
\end{eqnarray*}
\end{itemize}
 \end{thm}

\begin{proof}   We point out that
  $$
  0\leq  K_{\exp(-t\Delta_{\Omega})}(x,y)\leq {1\over (4\pi t)^{n/2}} \exp\Big(-{|x-y|^2\over 4t}\Big)
  $$
  (see e.g., Example 2.18, \cite{D}).
  That is  the heat kernels  corresponding to $\Delta_{\Omega}$
 satisfying Gaussian bounds  \eqref{GE} with $m=2$.

 Denote $T={ (I+\Delta_\Omega)^{-s }e^{it\Delta_\Omega}}$  with  $s=n/2$ for $p=1$ and  $s\geq n|{1/p-1/2}| $ for $ 1<p<\infty$. And denote by $\widetilde T$ the extension of
 $T$ as in \eqref{TT}.
 Then (i) and (ii) of Theorem~\ref{prop4.1} follow  from
 Theorem~\ref{th1.1}  and \cite[Theorem 1.1]{CDLY} by first applying to the extended operator ${\widetilde{ T}}$ and then restricting to
$T$, respectively.
This completes the proof of Theorem~\ref{prop4.1}.
\end{proof}

 \bigskip

\noindent
{\bf Acknowledgements}:
P. Chen is supported by NNSF of China 11501583, Guangdong Natural Science Foundation
2016A030313351.
Duong and Li are supported by the Australian Research Council (ARC) through the
research grants DP160100153 and  DP190100970 and by Macquarie University Research Seeding Grant.
L. Song is supported by  NNSF of China (No. 11622113) and NSF for distinguished Young Scholar of Guangdong Province (No. 2016A030306040).
 Yan is supported by the NNSF
of China, Grant No. ~11521101 and  ~11871480, and Guangdong Special Support Program.

 \end{document}